\documentclass[11pt,reqno]{amsart}

\usepackage[T1]{fontenc}
\usepackage[utf8]{inputenc}
\usepackage{amsmath,amssymb,mathtools}
\usepackage{newpxtext,newpxmath}
\usepackage[letterpaper,top=1.0in,bottom=0.98in,left=1.16in,right=1.16in,
            headheight=15pt]{geometry}
\usepackage[final]{microtype}
\usepackage{enumitem}
\usepackage{cite}
\usepackage{titlesec}
\usepackage{hyperref}
\hypersetup{
  hidelinks,
  pdftitle={Foliations and Laminar Currents for C2 Maximal Plurisubharmonic Functions},
  pdfauthor={Yizhe Yang, Zhanpeng You}
}
\titleformat{\section}[hang]
  {\normalfont\large\bfseries}{\thesection.}{0.68em}{}
\titlespacing*{\section}{0pt}{3.6ex plus 0.8ex minus 0.25ex}{2.5ex plus 0.3ex}
\titleformat{\subsection}[hang]
  {\normalfont\normalsize\bfseries}{\thesubsection.}{0.65em}{}
\titlespacing*{\subsection}{0pt}{2.6ex plus 0.5ex minus 0.2ex}{1.1ex plus 0.2ex}
\numberwithin{equation}{section}
\newtheoremstyle{plainblack}
  {1.0ex}{1.0ex}{\itshape}{}{\bfseries}{.}{0.55em}{}
\theoremstyle{plainblack}
\newtheorem{theorem}{Theorem}[section]
\newtheorem{lemma}[theorem]{Lemma}
\newtheorem{proposition}[theorem]{Proposition}
\newtheorem{corollary}[theorem]{Corollary}
\theoremstyle{definition}
\newtheorem{definition}[theorem]{Definition}
\theoremstyle{remark}
\newtheorem{remark}[theorem]{Remark}

\newcommand{\C}{\mathbb C}
\newcommand{\R}{\mathbb R}
\newcommand{\D}{\mathbb D}
\newcommand{\T}{\mathbb T}
\newcommand{\ddc}{dd^c}
\DeclareMathOperator{\PSH}{PSH}
\DeclareMathOperator{\Lip}{Lip}
\DeclareMathOperator{\rank}{rank}
\DeclareMathOperator{\supp}{supp}
\DeclareMathOperator{\loc}{loc}
\newcommand{\area}{\,dA}
\newcommand{\norm}[1]{\left\lVert#1\right\rVert}

\newcommand{\hb}{\widehat H}
\newcommand{\kb}{\widehat K}

\begin{document}

\title[Foliations and laminar currents for $C^2$ maximal psh functions]
{Foliations and Laminar Currents for
$C^2$ Maximal Plurisubharmonic Functions}
\author[Y. Yang]{Yizhe Yang}
\address{School of Mathematical Sciences, University of Chinese Academy of Sciences,
Beijing 100049, P. R. China}
\email{yangyizhe24@mails.ucas.ac.cn}

\author[Z. You]{Zhanpeng You}
\address{School of Mathematical Sciences, University of Chinese Academy of Sciences,
Beijing 100049, P. R. China}
\email{youzhanpeng23@mails.ucas.ac.cn}

\begin{abstract}
Let $u\in C^2(\Omega)\cap\PSH(\Omega)$ satisfy
$(\ddc u)^2=0$ on a domain $\Omega\subset\C^2$.
We prove that the positive-rank locus of its Levi form carries a
foliation by holomorphic curves tangent to the Levi-null directions,
along which $u$ is harmonic. The flow boxes are locally bi-Lipschitz;
in particular, the Levi-null slope is locally Lipschitz in the
transverse variable. The Levi current is uniformly laminar on the positive-rank locus,
with an explicit transverse measure, and laminar on $\Omega$. A smooth example also shows that the Levi current can be laminar
without being uniformly laminar on the whole domain. The analytic
construction combines a local partial Legendre transform with a
Hardy-space boundary equation.
\end{abstract}

\maketitle

\section{Introduction}\label{sec:introduction}

The homogeneous complex Monge--Amp\`ere equation provides a natural
link between pluripotential theory and the geometry of holomorphic
curves. Let $\Omega\subset\C^2$ and let
$u\in C^2(\Omega)\cap\PSH(\Omega)$ satisfy
\begin{equation}\label{eq:hcma-intro}
 (\ddc u)^2=0.
\end{equation}
At every point where the Levi form of $u$ is nonzero, it has complex
rank one, and its kernel defines a complex line field. A natural
geometric question is whether this line field is tangent to a
foliation by holomorphic curves along which $u$ is harmonic.
Closely related is the laminar structure of the positive closed
current $T=\ddc u$: under what regularity assumptions is $T$
laminar or uniformly laminar?

For more regular potentials, Bedford and Kalka \cite{BK} established
the relationship between the homogeneous complex Monge--Amp\`ere
equation and holomorphic curve foliations. If $u\in C^3$, the
Levi-null distribution is $C^1$ on the positive-rank locus, and
Frobenius theory yields a foliation by holomorphic curves along
which $u$ is harmonic. The associated Levi current admits a
uniformly laminar representation on this locus. Further work on
Monge--Amp\`ere foliations includes \cite{Freeman,KP,CT,RWN}.

At $C^2$ regularity, Kruzhilin \cite{Kru} established a foliation
theorem for rotationally symmetric potentials on complete Hartogs
domains. More precisely, assuming that the Levi form is nowhere
zero and $u(z,w)=u(|z|,w)$, he proved the existence of a foliation
by holomorphic curves along which $u$ is harmonic. Dujardin and
Guedj \cite[Question~3.17]{DG} asked whether the Levi current of a
general $C^2$ maximal plurisubharmonic potential must be uniformly
laminar.

The main difficulty in the general $C^2$ case is that the Levi-null
line field is only continuous. Indeed, in local coordinates with
$u_{z\bar z}>0$, a holomorphic graph $z=h(w)$ is tangent to this
field precisely when
\begin{equation}\label{eq:leafode-intro}
 h'(w)=-\overline{\left(\frac{u_{z\bar w}}{u_{z\bar z}}\right)(h(w),w)}.
\end{equation}
The right-hand side is in general merely continuous, so the
classical Frobenius theorem does not directly supply holomorphic
solutions.

At lower regularity, Sibony's construction yields maximal
plurisubharmonic potentials of class $C^{1,1}$ whose Levi currents
are not uniformly laminar; whether these currents are laminar
remains unknown (see \cite{DG,Dujardin}). Dujardin
\cite{Dujardin} constructed maximal potentials of class
$C^{1,\alpha}$ for every $0<\alpha<1$, whose Levi currents are
not laminar: their supports contain no holomorphic disks.
These examples underscore the role of regularity in the geometry
of Levi currents.

We study the Levi currents of general $C^2$ maximal
plurisubharmonic potentials. Our main result gives a holomorphic
foliation of the positive-rank locus with quantitative transverse
regularity and an associated uniformly laminar current
representation. We use the convention $\ddc=i\partial\bar\partial$.

\begin{theorem}\label{thm:main}
Let $u\in C^2(\Omega)\cap\PSH(\Omega)$ satisfy
$(\ddc u)^2=0$, and let
\[
 T=\ddc u,\qquad
 \mathcal R=\{x\in\Omega:T(x)\ne0\},\qquad
 \mathcal Z=\Omega\setminus\mathcal R.
\]
Then the following hold.
\begin{enumerate}[label=\textup{(\roman*)},leftmargin=2.3em]
\item \emph{Foliation and uniqueness.}
For every $p\in\mathcal R$, there are local holomorphic coordinates
$(z,w)$ centered at $p$ and a neighborhood $W\Subset\mathcal R$
admitting a locally bi-Lipschitz flow box
\[
 \Phi:\Delta_s\times\Delta_{\rho/2}\longrightarrow W,
 \qquad \Phi(\xi,w)=(h_\xi(w),w),
\]
where $\Delta_r=\{\zeta\in\C:|\zeta|<r\}$, and the
$h_\xi$ are holomorphic on $\Delta_\rho$, satisfy
$h_\xi(0)=\xi$, and have pairwise disjoint graphs over
$\Delta_{\rho/2}$. These plaques are tangent to $\ker T$, and
$u$ is harmonic on each plaque. Through each point of $\mathcal R$
there is a unique germ of a nondegenerate holomorphic curve on which
$u$ is harmonic.
\item \emph{Transverse Lipschitz regularity.}
On a smaller positive-rank flow box with $a=u_{z\bar z}>0$,
the coefficient $q=u_{z\bar w}/u_{z\bar z}$ is locally
Lipschitz in the transverse $z$ variable, uniformly on smaller
horizontal slices: for a suitable constant $C>0$,
\[
 |q(z,w)-q(z',w)|\le C|z-z'|.
\]
\item \emph{Uniform laminarity and transverse measure.}
In each such flow box, writing
$\Gamma_\xi=\{(h_\xi(w),w):|w|<\rho/2\}$ and
$a=u_{z\bar z}>0$ in the adapted coordinates, one has
\begin{equation}\label{eq:currentintro}
 T|_W=\int_{\Delta_s}[\Gamma_\xi] \,2a(\xi,0)\,dA(\xi).
\end{equation}
Here $dA$ is the Lebesgue area measure on $\C$.
In particular, $T|_{\mathcal R}$ is uniformly laminar,
subordinate to the positive-rank foliation, and these transverse
measures are invariant under the holonomy of the foliation.
\end{enumerate}
\end{theorem}

\begin{corollary}\label{cor:fullmass}
Under the hypotheses of Theorem~\ref{thm:main}, the positive closed
current $T=\ddc u$ is laminar on $\Omega$ in the sense of
Definition~\ref{def:laminar}.
\end{corollary}

The laminarity conclusion cannot be strengthened to uniform
laminarity in general, as the following proposition shows.

\begin{proposition}
\label{prop:nonuniform}
There exists a smooth maximal plurisubharmonic function $u$
on a bidisk $\Omega\subset\C^2$ such that $\ddc u$ is not
uniformly laminar.
\end{proposition}

The basic idea of our proof comes from the work of Berndtsson,
Cordero-Erausquin, Klartag, and Rubinstein \cite{BCCKR} on complex
interpolation of convex functions, Legendre duality, and holomorphic
foliations. An important insight of their work is that Legendre
duality connects the complex Monge--Amp\`ere equation with the
construction of holomorphic leaves. Whereas their approach starts
from globally convex boundary data and studies the associated
interpolation potential, we work directly with a prescribed local
$C^2$ maximal plurisubharmonic potential. We first choose local
coordinates in which the potential is strictly convex in the real
transverse variables, and then introduce the partial Legendre dual
\[
 v(p,w)=\sup_z\{2\operatorname{Re}(pz)-u(z,w)\},
\]
where the supremum is taken over an appropriate local neighborhood.
The homogeneous complex Monge--Amp\`ere equation ensures that the
dual function remains plurisubharmonic. Thus the Fenchel inequality
\[
 u(z,w)+v(p,w)-2\operatorname{Re}(pz)\ge0
\]
provides the key positivity structure for constructing holomorphic
harmonic leaves.

More precisely, we use a contraction fixed-point equation in a
Hardy space to construct a pair of holomorphic functions satisfying
a boundary gradient relation. The Fenchel inequality and the
maximum principle for subharmonic functions then propagate equality
from the boundary to the whole disk, proving that the original
potential is harmonic along the resulting holomorphic curves.
Further quantitative estimates yield locally bi-Lipschitz flow
boxes and uniqueness of the leaves, while closedness of the Levi
form determines the transverse measures and hence the laminar
representation of the Levi current.

\section{Preliminaries}\label{sec:definitions}

We set
\[
 \omega_0=\frac i2\bigl(dz\wedge d\bar z+dw\wedge d\bar w\bigr),
 \qquad \D=\Delta_1,\qquad \T=\partial\D,
\]
and write $dm(t)=dt/(2\pi)$ for normalized arc-length measure
on $\T$. The notation $dA$ denotes Lebesgue area measure on $\C$.

We recall the notions of a foliation by holomorphic curves, an
embedded lamination, and the corresponding notions for positive
closed currents. Our terminology follows \cite[Definitions~3.9
and~3.14]{DG}.

\begin{definition}\label{def:foliation}
A \emph{foliation by holomorphic curves} on an open set
$V\subset\C^2$ is an atlas of compatible flow-box
homeomorphisms $\Phi:\Delta\times\Delta\to W\subset V$ such
that $w\mapsto\Phi(\xi,w)$ is a holomorphic embedding for each
fixed transverse parameter $\xi$; transition maps preserve plaques.
We call such a flow box \emph{locally bi-Lipschitz} if it and its
inverse are locally Lipschitz. A \emph{$C^1$ foliation} instead
has $C^1$ flow-box diffeomorphisms, as in the classical Frobenius
setting.

An \emph{embedded lamination} is the analogous compatible family
of holomorphic plaques on a relatively closed set $X\subset V$;
locally the transverse disk is replaced by a relatively closed
$K\subset\Delta$ and the flow box identifies $\Delta\times K$
with $X\cap W$. In particular, $X$ need not contain an open set.
\end{definition}

\begin{definition}\label{def:foliatedcycle}
A \emph{foliated cycle} subordinate to a foliation is a positive
closed $(1,1)$-current given in each foliation flow box by
integrating over its plaques against a locally finite positive
transverse measure invariant under holonomy.
\end{definition}

\begin{definition}\label{def:uniform}
A positive closed $(1,1)$-current $S$ on $V$ is
\emph{uniformly laminar} if it is subordinate to an embedded
lamination and, near each point of its support, there are pairwise
disjoint holomorphic graphs over a common larger disk and a positive
transverse measure $\nu$ such that on a smaller flow box $W$,
\begin{equation}\label{eq:defuniform}
 S|_W=\int_K[\Gamma_a\cap W] \,d\nu(a).
\end{equation}
The local transverse measures are compatible under holonomy.
Restricting the parameter set to $\supp\nu$ does not change the
current, and the corresponding plaques cover $\supp S$ locally.
The common larger disk is understood locally, after shrinking the
flow box when necessary. This is the support-point formulation of
\cite[Definition~3.9]{DG}.
\end{definition}

\begin{definition}\label{def:laminar}
A positive closed $(1,1)$-current $S$ on $V$ is
\emph{laminar} if, for every relatively compact open set
$U\Subset V$ and every $\varepsilon>0$, there exist an open
$U_\varepsilon\subset U$ and a locally uniformly laminar positive
closed current $S_\varepsilon$ on $U_\varepsilon$ whose extension
by zero to $U$ defines a positive current
$\widetilde S_\varepsilon$ such that
\begin{equation}\label{eq:deflaminar}
 0\le\widetilde S_\varepsilon\le S|_U,
 \qquad
 \int_U(S|_U-\widetilde S_\varepsilon)\wedge\omega_0
 <\varepsilon.
\end{equation}
Here $S_\varepsilon$ is positive and closed on its own open domain
$U_\varepsilon$; only its zero extension as a positive current,
not closedness of that extension across $\partial U_\varepsilon$,
is required. This is the local trace-mass exhaustion form of
\cite[Definition~3.14]{DG}, used to accommodate currents of
infinite total mass on noncompact domains.
\end{definition}

For currents of finite trace mass, the local trace-mass
formulation is equivalent to the usual mass-exhaustion formulation
of \cite{DG}. Uniform laminarity is expressed by a common embedded
lamination, while laminarity is measured by local approximation in
trace mass.

\section{Levi-null geometry and transverse convexification}
\label{sec:normalization}
\subsection{The null direction at a positive-rank point}

Let $\Omega\subset\C^2$ have coordinates $(z,w)$ and let
$u\in C^2(\Omega)$ be real valued. Put
\begin{equation}\label{eq:H}
 H_u=\begin{pmatrix}a&b\\\bar b&c\end{pmatrix},\qquad
 a=u_{z\bar z},\quad b=u_{z\bar w},\quad c=u_{w\bar w}.
\end{equation}
The conditions $u\in\PSH(\Omega)$ and $(\ddc u)^2=0$ are equivalent
to $H_u\ge0$ and $ac-|b|^2=0$, respectively. If $a>0$, then $H_u$ has
rank one. Setting $q=b/a$, for a holomorphic graph $w\mapsto(h(w),w)$
we have
\begin{equation}\label{eq:graphLevi}
 \partial_w\partial_{\bar w}\bigl(u(h(w),w)\bigr)
 =a\left|h'(w)+\bar q\right|^2,
\end{equation}
where the coefficients on the right are evaluated at $(h(w),w)$.
Consequently, such a graph is harmonic for $u$ precisely when
\begin{equation}\label{eq:leafode}
 h'(w)=-\overline{q(h(w),w)}.
\end{equation}
For a holomorphic graph, equation~\eqref{eq:graphLevi} therefore
identifies the Levi-null condition with the first-order equation
\eqref{eq:leafode}. The coefficient $q=b/a$ is only continuous; the
construction of holomorphic solutions cannot be inferred from a
classical differentiable Frobenius theorem.

\subsection{A normal form with small transverse derivatives}

\begin{lemma}\label{lem:norm}
At a point $p_0\in\mathcal R$, a holomorphic linear
change of coordinates and subtraction of a real-valued pluriharmonic
polynomial of degree at most two, followed by an adjustment of the
constant, reduce the local potential to
\begin{equation}\label{eq:normal}
 u(z,w)=|z|^2+o(|z|^2+|w|^2).
\end{equation}
For $E=u_z-\bar z$ we then have $E\in C^1$ and
\begin{equation}\label{eq:Ezero}
 E(0)=0,\qquad DE(0)=0.
\end{equation}
These operations preserve harmonicity of the restriction of $u$ to
any holomorphic curve.
\end{lemma}

\begin{proof}
The nonzero Hermitian matrix $H_u(p_0)$ is positive semidefinite of
rank one. A complex linear coordinate change puts it into the form
$\operatorname{diag}(1,0)$. The second-order Taylor polynomial of a
real-valued function is the sum of its Hermitian quadratic part, the
real part of a holomorphic quadratic polynomial, a real linear term,
and a constant. Subtracting the latter three terms leaves precisely
$|z|^2$ as the quadratic part, proving \eqref{eq:normal}.
The $C^2$ regularity gives \eqref{eq:Ezero}. The restriction of the
real part of a holomorphic function to a holomorphic curve is harmonic,
so the subtracted terms do not change the leaves under consideration.
We continue to denote the normalized potential by $u$.
\end{proof}

Fix, once and for all,
\begin{equation}\label{eq:lambda}
 \lambda=\frac1{100}.
\end{equation}

\begin{lemma}
\label{lem:extension}
There are $R>0$ and $\widetilde E\in C_c^1(\C^2,\C)$ such that
\begin{equation}\label{eq:Eextend}
 \widetilde E=E\quad\text{on }\Delta_R\times\Delta_R,
 \qquad\widetilde E(0)=0,
 \qquad\Lip(\widetilde E)\le\lambda.
\end{equation}
In particular, setting $\sigma=u_{zz}$ on this bidisk, we have
\begin{equation}\label{eq:convexcondition}
 a\ge1-\lambda,\qquad |\sigma|\le\lambda,
 \qquad a-|\sigma|\ge1-2\lambda>0.
\end{equation}
\end{lemma}

\begin{proof}
Choose a fixed smooth cutoff equal to $1$ on the real ball $B(0,2)$
and supported in $B(0,3)$. Because $DE(0)=0$, we may choose $R>0$
such that $\sup_{B(0,4R)}|DE|\le\delta$, where $\delta$ will be
chosen depending only on the cutoff. Since $E(0)=0$, we have
$|E(x)|\le\delta|x|$. Multiply $E(x)$ by the cutoff
$\chi(x/R)$ and extend the product by zero outside the ball.
This gives a $C_c^1$ extension. In the product rule, the two
derivative terms are bounded, respectively, by $\delta$ and
$R^{-1}\cdot3\delta R\norm{D\chi}_\infty$. Choosing $\delta$
sufficiently small gives \eqref{eq:Eextend}.

Since $\Delta_R^2\subset B(0,2R)$, the extension agrees with $E$
there. The identities $E_z=\sigma$ and $E_{\bar z}=a-1$,
together with the fact that the modulus of each Wirtinger derivative
is bounded by the norm of the real differential, yield
\eqref{eq:convexcondition}. For every $v\in\C$ the transverse real
Hessian satisfies
\[
 D_z^2u[v,v]=2a|v|^2+2\operatorname{Re}(\sigma v^2)
 \ge2(a-|\sigma|)|v|^2.
\]
Thus $z\mapsto u(z,w)$ is strictly convex on $\Delta_R$ for each
fixed $w$ in the bidisk.
\end{proof}

\section{A plurisubharmonic partial Legendre dual}
\label{sec:duality}

The role of the partial Legendre transform is not merely to
reparametrize the null direction. In the $z$ variable the real
Hessian of $u$ is strictly positive by \eqref{eq:convexcondition}, so
one can solve the real-gradient equation $p=u_z$ uniquely. The
homogeneous Monge--Amp\`ere equation then has the decisive effect
that the resulting dual weight is itself plurisubharmonic, allowing
a maximum-principle argument on holomorphic disks. We dualize only in
the transverse $z$ variable, not in the null $w$ direction.

The dual variable is denoted by $p\in\C$ and should not be confused
with the point $p_0$. We use the real bilinear pairing
$2\operatorname{Re}(pz)$. With this convention the stationary equation
is $p=u_z$, not $p=2u_z$.

\begin{lemma}
\label{lem:inverse}
For every $(p,w)\in\Delta_{R/4}\times\Delta_{R/4}$, the equation
\begin{equation}\label{eq:stationary}
 p=u_z(z,w)=\bar z+E(z,w)
\end{equation}
has a unique solution $z=Z(p,w)$ in $\Delta_{R/2}$, and $Z\in C^1$.
\end{lemma}

\begin{proof}
Write the equation as
\[
 z=\bar p-\overline{E(z,w)}.
\]
On $|z|\le R/2$, the right-hand side has Lipschitz constant at
most $\lambda$ as a function of $z$, and
\[
 |\bar p-\overline{E(z,w)}|
 \le\frac R4+\lambda\left(\frac R2+\frac R4\right)
 <\frac R2.
\]
The Banach fixed-point theorem therefore gives a unique solution.
Equivalently, this fixed-point equation inverts the real map
$F_w(z):=z+\overline{E(z,w)}=u_{\bar z}(z,w)$ at the value
$\bar p$. Its real Jacobian determinant is
$\det_{\R}D_zF_w=a^2-|\sigma|^2>0$.
The parameter-dependent real inverse function theorem therefore
gives local inverses of class $C^1$; by uniqueness, they patch
together to form the asserted inverse $Z(p,w)$ on the common domain.
\end{proof}

Define the partial real Legendre transform by
\begin{equation}\label{eq:legendredef}
 v(p,w)=pZ(p,w)+\bar p\,\overline{Z(p,w)}-u(Z(p,w),w).
\end{equation}

\begin{lemma}\label{lem:fenchel}
The function $v$ belongs to $C^2(\Delta_{R/4}^2)$ and satisfies
\begin{equation}\label{eq:envelopefirst}
 v_p=Z,\qquad v_w=-u_w(Z(p,w),w).
\end{equation}
For $|z|<R$ and $|p|,|w|<R/4$, one has
\begin{equation}\label{eq:fenchel}
 u(z,w)+v(p,w)-2\operatorname{Re}(pz)\ge0,
\end{equation}
with equality if and only if $z=Z(p,w)$, or equivalently
$p=u_z(z,w)$.
\end{lemma}

\begin{proof}
When \eqref{eq:legendredef} is differentiated once, all terms
containing $dZ$ cancel because $p=u_z(Z,w)$. This gives
\eqref{eq:envelopefirst} and its conjugate. The right-hand sides
are of class $C^1$, hence $v\in C^2$.

For fixed $w$, the function
$z\mapsto u(z,w)-2\operatorname{Re}(pz)$ is strictly convex on the
convex disk $\Delta_R$ and has vanishing first real differential
at $Z(p,w)\in\Delta_{R/2}$. It attains its unique minimum there.
This is precisely \eqref{eq:fenchel}, together with its equality
condition.
\end{proof}

\begin{lemma}\label{lem:dualle}
The function $v$ is plurisubharmonic and satisfies $(\ddc v)^2=0$.
More precisely, write
\[
 \begin{gathered}
 a=u_{z\bar z},\quad b=u_{z\bar w},\quad c=u_{w\bar w},
 \quad\sigma=u_{zz},\quad t=u_{zw},\\
 D=a^2-|\sigma|^2,\qquad A=at-\sigma\bar b.
 \end{gathered}
\]
All coefficients on the right are evaluated at $(Z(p,w),w)$.
The strict convexity of $u$ in the real transverse variable gives
$D>0$; the Monge--Amp\`ere identity $ac=|b|^2$ is responsible for
the rank-one factorization below:
\begin{equation}\label{eq:dualH}
 H_v=\frac1{aD}
 \begin{pmatrix}
 a^2&-a\bar A\\
 -aA&|A|^2
 \end{pmatrix}
 =\frac1{aD}
 \begin{pmatrix}a\\-A\end{pmatrix}
 \begin{pmatrix}a&-\bar A\end{pmatrix}\ge0.
\end{equation}
\end{lemma}

\begin{proof}
The identities $p=u_z$ and $\bar p=u_{\bar z}$ give
\begin{align*}
 dp&=\sigma\,dz+a\,d\bar z+t\,dw+b\,d\bar w,\\
 d\bar p&=a\,dz+\bar\sigma\,d\bar z+\bar b\,dw+\bar t\,d\bar w.
\end{align*}
The inverse of the transverse coefficient matrix is
\[
 \begin{pmatrix}\sigma&a\\a&\bar\sigma\end{pmatrix}^{-1}
 =\frac1D\begin{pmatrix}-\bar\sigma&a\\a&-\sigma\end{pmatrix}.
\]
Consequently,
\begin{align}
 Z_{\bar p}&=\frac aD,\label{eq:inv1}\\
 Z_{\bar w}&=\frac{\bar\sigma b-a\bar t}{D},\label{eq:inv2}\\
 \overline Z_{\bar w}&=\frac{\sigma\bar t-ab}{D}.
 \label{eq:inv3}
\end{align}
Combining these formulas with \eqref{eq:envelopefirst} yields
\begin{align}
 v_{p\bar p}&=\frac aD,\label{eq:vpp}\\
 v_{p\bar w}&=\frac{\bar\sigma b-a\bar t}{D},\label{eq:vpw}\\
 v_{w\bar w}
 &=-c+\frac{a(|t|^2+|b|^2)
       -2\operatorname{Re}(\bar\sigma bt)}{D}.
 \label{eq:vww}
\end{align}
Here $D>0$ follows from $a>|\sigma|$ in
\eqref{eq:convexcondition}. Using the homogeneous Monge--Amp\`ere
equation $ac=|b|^2$ in \eqref{eq:vww} gives
\begin{align*}
 v_{w\bar w}
 &=\frac{a^2|t|^2+|\sigma|^2|b|^2
       -2a\operatorname{Re}(\bar\sigma bt)}{aD}\\
 &=\frac{|at-\sigma\bar b|^2}{aD}.
\end{align*}
Moreover, the numerator in \eqref{eq:vpw} is $-\bar A$.
This proves \eqref{eq:dualH}. The matrix is a positive multiple of
a positive semidefinite rank-one matrix. \end{proof}

\begin{remark}
The same inverse-function calculation, without assuming
$(\ddc u)^2=0$, gives the identity
\[
 \det H_v=-\frac{\det H_u}{a^2-|u_{zz}|^2}.
\]
Thus the homogeneous equation is essential for the dual to be
plurisubharmonic: for $u=|z|^2+|w|^2$ one obtains
$v=|p|^2-|w|^2$. The Hardy boundary equation is localized in
Lemma~\ref{lem:localization}.
\end{remark}

\section{Holomorphic disks: Hardy data and the Fenchel gap}
\label{sec:hardy}

To obtain an actual holomorphic integral curve, we seek holomorphic
functions $H$ and $K$ on a common disk whose boundary values obey
$K=u_z(H,w)$. Since the transverse gradient has the form
$u_z=\bar z+E(z,w)$, this is an analytic-disc boundary problem for
the nonholomorphic perturbation $E$. Orthogonal Fourier projection
turns it into a contraction on a Hardy space. An $H^2$ fixed point,
however, has only $L^2$ boundary values and need not remain in the
local domain of the Legendre transform. Rotational difference
quotients will improve the boundary data to $W^{1,2}$, giving the
uniform boundary control that the subsequent Fenchel argument needs.

\subsection{Fourier projections and boundary estimates}

Write a function in $L^2(\T)$ as
$f(e^{it})=\sum_{n\in\mathbb Z}f_ne^{int}$ and put
\[
 H^2_0=\{f\in L^2(\T):f_n=0\text{ for }n\le0\}.
\]
Let $P_-f=\sum_{n<0}f_ne^{int}$. It is an orthogonal projection
of norm one, while complex conjugation maps the negative-frequency
space isometrically onto $H^2_0$.

If $g\in H^2_0$, its holomorphic extension to $\D$ satisfies
\begin{equation}\label{eq:hardypoint}
 |g(\zeta)|\le
 \frac{|\zeta|}{\sqrt{1-|\zeta|^2}}\norm{g}_{L^2(\T)}.
\end{equation}
Indeed, apply Cauchy--Schwarz directly to
$\sum_{n\ge1}g_n\zeta^n$. Similarly,
\begin{equation}\label{eq:hardyder}
 |g'(\zeta)|
 \le\left(\frac{1+|\zeta|^2}{(1-|\zeta|^2)^3}\right)^{1/2}
 \norm{g}_{L^2(\T)}.
\end{equation}
If also $g\in W^{1,2}(\T)$, then
\begin{equation}\label{eq:hardyboundary}
 \norm{g}_{L^\infty(\T)}
 \le\sum_{n\ge1}|g_n|
 \le\frac\pi{\sqrt6}\norm{\partial_tg}_{L^2(\T)}.
\end{equation}
In particular, its Fourier series converges absolutely and uniformly
on the closed disk, yielding a holomorphic function continuous up to
the boundary.

\begin{lemma}
\label{lem:hardy}
Suppose $\widetilde E:\C^2\to\C$ satisfies the Lipschitz bound in
\eqref{eq:Eextend}. For any $\xi\in\C$ and $\rho>0$, there is a
unique $g_\xi\in H^2_0$ such that
\begin{equation}\label{eq:fixed}
 g_\xi=-\overline{P_-
       \widetilde E(\xi+g_\xi,\rho e^{it})}.
\end{equation}
Let $\hb_\xi=\xi+g_\xi$. There is a holomorphic function
$\kb_\xi$, continuous on $\overline\D$, satisfying
\begin{equation}\label{eq:boundaryEK}
 \kb_\xi(e^{it})
 =\overline{\hb_\xi(e^{it})}
   +\widetilde E(\hb_\xi(e^{it}),\rho e^{it}).
\end{equation}
Moreover, with
\begin{equation}\label{eq:Bdef}
 B=\frac\pi{\sqrt6}\frac\lambda{1-\lambda}<\frac1{50},
\end{equation}
we have
\begin{equation}\label{eq:gBound}
 \norm{\partial_tg_\xi}_2\le\frac{\lambda\rho}{1-\lambda},
 \qquad\norm{g_\xi}_{L^\infty(\overline\D)}\le B\rho,
\end{equation}
and the parameter estimate
\begin{equation}\label{eq:paramH2}
 \norm{g_\xi-g_\eta}_2
 \le\frac\lambda{1-\lambda}|\xi-\eta|.
\end{equation}
\end{lemma}

\begin{proof}
\emph{Existence by contraction.}
The right-hand side of \eqref{eq:fixed} maps the complete space
$H^2_0$ into itself. As a function $\mathcal F_\xi$ of $g$ it
satisfies
\[
 \norm{\mathcal F_\xi(g)-\mathcal F_\xi(\widetilde g)}_2
 \le\lambda\norm{g-\widetilde g}_2.
\]
Although complex conjugation is antilinear, it preserves distance;
the Banach fixed-point theorem applies to the underlying real metric
space. The global Lipschitz property also ensures that the
composition with $\widetilde E$ belongs to $L^2$. Thus there is a
unique fixed point.

\smallskip
\noindent\emph{Rotational difference quotients.}
Define $R_\theta g(e^{it})=g(e^{i(t+\theta)})$.
Both $P_-$ and pointwise complex conjugation commute with
$R_\theta$. Rotating \eqref{eq:fixed} and subtracting the original
equation yields
\[
 \norm{R_\theta g_\xi-g_\xi}_2
 \le\lambda\norm{R_\theta g_\xi-g_\xi}_2
       +\lambda\rho|e^{i\theta}-1|.
\]
Therefore
\begin{equation}\label{eq:rotbound}
 \norm{R_\theta g_\xi-g_\xi}_2
 \le\frac{\lambda\rho}{1-\lambda}|e^{i\theta}-1|.
\end{equation}
Writing $g_\xi=\sum_{n\ge1}c_ne^{int}$, Parseval's identity and
Fatou's lemma give
\begin{align*}
 \sum_{n\ge1}n^2|c_n|^2
 &\le\liminf_{\theta\to0}
   \sum_{n\ge1}\frac{|e^{in\theta}-1|^2}{|\theta|^2}|c_n|^2\\
 &\le\left(\frac{\lambda\rho}{1-\lambda}\right)^2.
\end{align*}
Hence $g_\xi\in W^{1,2}(\T)$, and the derivative estimate in
\eqref{eq:gBound} follows. Inequality \eqref{eq:hardyboundary}
gives the remaining bound. This is an actual boundary estimate at
the fixed radius $\rho$; the $H^2$ estimate alone would not
suffice.

\smallskip
\noindent\emph{The second holomorphic function.}
Define the boundary function by the right-hand side of
\eqref{eq:boundaryEK}. The Sobolev chain rule, together with
$g_\xi\in W^{1,2}(\T)$, shows that it belongs to
$W^{1,2}(\T)$. Since
$P_-(\overline{\hb_\xi})=\bar g_\xi$, the conjugate of
\eqref{eq:fixed} says exactly that $P_-\kb_\xi=0$.
Thus its Fourier series has only nonnegative frequencies and
converges absolutely and uniformly on the closed disk. It defines
the required continuous holomorphic extension.

\smallskip
\noindent\emph{Dependence on the transverse parameter.}
Subtracting the fixed-point equations for $\xi$ and $\eta$ gives
\[
 \norm{g_\xi-g_\eta}_2
 \le\lambda\bigl(|\xi-\eta|+\norm{g_\xi-g_\eta}_2\bigr),
\]
which is \eqref{eq:paramH2}.
\end{proof}

\subsection{Vanishing of the Fenchel gap}\label{sec:fenchel-gap}

The essential step is to propagate a boundary equality to the full
disk without differentiating a merely continuous null direction.
For the original and dual psh weights the Fenchel gap is nonnegative;
along the pair of holomorphic graphs it is also subharmonic. Once the
Hardy boundary equation makes the gap vanish on the circle, the
maximum principle forces equality everywhere inside. Fix
\begin{equation}\label{eq:radii}
 s=\rho=\frac R{32}.
\end{equation}
For each $\xi\in\overline{\Delta_s}$, apply Lemma~\ref{lem:hardy}.
By \eqref{eq:gBound}, throughout the closed disk,
\begin{equation}\label{eq:Hrange}
 |\hb_\xi|\le s+B\rho=(1+B)s<\frac R8.
\end{equation}
Furthermore, \eqref{eq:boundaryEK}, $\widetilde E(0)=0$, and the
Lipschitz estimate show that on the circle,
\[
 |\kb_\xi|
 \le(1+\lambda)|\hb_\xi|+\lambda\rho
 \le\bigl((1+\lambda)(1+B)+\lambda\bigr)s<\frac R8.
\]
By the maximum modulus principle, this bound holds on the whole
closed disk. Consequently,
\begin{equation}\label{eq:Kranges}
 (\hb_\xi(\zeta),\rho\zeta)\in\Delta_R^2,
 \qquad
 (\kb_\xi(\zeta),\rho\zeta)\in\Delta_{R/4}^2
 \quad(|\zeta|\le1).
\end{equation}
Thus both the original potential $u$ and its local dual $v$ are
defined along the relevant holomorphic graphs.

\begin{lemma}[Localization of the Hardy boundary equation]
\label{lem:localization}
For every $|\xi|\le s$, the holomorphic pair
$(\hb_\xi,\kb_\xi)$ constructed using $\widetilde E$ takes values
in the domains where $u$ and $v$ are defined, and satisfies
\begin{equation}\label{eq:trueboundary}
 \kb_\xi(e^{it})=u_z(\hb_\xi(e^{it}),\rho e^{it})
 \quad(t\in\R).
\end{equation}
Consequently the Fenchel inequality for the \emph{original}
$u$ and its local dual applies along these disks, with equality
on their common boundary.
\end{lemma}

\begin{proof}
The fixed point in Lemma~\ref{lem:hardy} was obtained using the
globally Lipschitz extension $\widetilde E$, rather than $u_z$.
The range bounds \eqref{eq:Hrange}--\eqref{eq:Kranges} first place
the primal and dual graphs in $\Delta_R^2$ and
$\Delta_{R/4}^2$, respectively. On the boundary of the disk,
$\widetilde E=E=u_z-\bar z$ at the primal graph, so
\eqref{eq:boundaryEK} gives \eqref{eq:trueboundary}. The equality
case of Lemma~\ref{lem:fenchel} now applies on that boundary;
its inequality holds throughout the disk by the same range bounds.
\end{proof}

\begin{proposition}
\label{prop:gap}
For every $\xi\in\overline{\Delta_s}$, the function
\begin{align}\label{eq:gap}
 F_\xi(\zeta)
 ={}&u(\hb_\xi(\zeta),\rho\zeta)
     +v(\kb_\xi(\zeta),\rho\zeta)\notag\\
    &-2\operatorname{Re}(\hb_\xi(\zeta)\kb_\xi(\zeta))
\end{align}
is continuous on the closed disk and identically zero on $\D$.
Consequently,
\begin{equation}\label{eq:interiorgrad}
 \kb_\xi(\zeta)=u_z(\hb_\xi(\zeta),\rho\zeta)
 \qquad(\zeta\in\D),
\end{equation}
and $\zeta\mapsto u(\hb_\xi(\zeta),\rho\zeta)$ is harmonic.
\end{proposition}

\begin{proof}
Lemma~\ref{lem:fenchel} gives $F_\xi\ge0$.
By Lemma~\ref{lem:dualle}, the first two terms in
\eqref{eq:gap} are subharmonic as compositions of plurisubharmonic
functions with holomorphic maps. The final term is the real part of
a holomorphic function and is therefore harmonic. Hence $F_\xi$
is subharmonic. The boundary identity of Lemma~\ref{lem:localization}
forces equality in the Fenchel inequality, so
$F_\xi|_{\T}=0$. The maximum principle gives $F_\xi\le0$.
It follows that
\[
 F_\xi\equiv0.
\]
The strict-convexity equality condition in the Fenchel inequality
then yields \eqref{eq:interiorgrad}.

The first two terms in \eqref{eq:gap} are $C^2$ subharmonic
functions, while their sum is now harmonic. Their Laplacians are
nonnegative and sum to zero, so each Laplacian vanishes. In
particular, the original potential is harmonic along the
constructed holomorphic graph.
\end{proof}

In the physical $w$ coordinate, set
\begin{equation}\label{eq:hphys}
 h_\xi(w)=\hb_\xi(w/\rho),\qquad |w|<\rho.
\end{equation}
Then $h_\xi(0)=\xi$, and
\begin{equation}\label{eq:harmonicfinal}
 \ddc_w\,u(h_\xi(w),w)=0.
\end{equation}
The derivative of the graph map $w\mapsto(h_\xi(w),w)$ is
$(h_\xi'(w),1)$, which never vanishes. Thus the resulting disks
are nondegenerate.

\begin{remark}[A translation model]
For $u(z,w)=|z+w^2|^2$ one has $u_z=\bar z+\bar w^2$ and
$v(p,w)=|p|^2-2\operatorname{Re}(pw^2)$. The boundary equation
gives $\hb_\xi(\zeta)=\xi-\rho^2\zeta^2$ and
$\kb_\xi(\zeta)=\bar\xi$. Thus
$h_\xi(w)=\xi-w^2$ and $u(h_\xi(w),w)=|\xi|^2$.
This checks the complex-conjugation convention in the
Legendre--Hardy construction.
\end{remark}

\section{Bi-Lipschitz foliations and uniqueness}
\label{sec:foliation}
\subsection{Disjointness, covering, and flow boxes}

By \eqref{eq:hardypoint} and \eqref{eq:paramH2}, for
$|\zeta|\le1/2$,
\begin{equation}\label{eq:gparampoint}
 |g_\xi(\zeta)-g_\eta(\zeta)|
 \le\kappa|\xi-\eta|,
 \qquad\kappa=\frac\lambda{\sqrt3(1-\lambda)}<\frac1{50}.
\end{equation}
Therefore, for $|w|\le\rho/2$,
\begin{equation}\label{eq:biholo}
 (1-\kappa)|\xi-\eta|
 \le|h_\xi(w)-h_\eta(w)|
 \le(1+\kappa)|\xi-\eta|.
\end{equation}
In particular, distinct graphs are disjoint, and the family
depends continuously on its transverse parameter.

\medskip
\noindent\emph{Covering a fixed bidisk.}
Fix $|z|<s/2$ and $|w|<\rho/2$. Finding the leaf through $(z,w)$
is equivalent to solving
\begin{equation}\label{eq:surjfixed}
 \xi=z-g_\xi(w/\rho).
\end{equation}
The right-hand side is $\kappa$-Lipschitz in $\xi$, and
\[
 |z-g_\xi(w/\rho)|\le s/2+B\rho<3s/4.
\]
It is therefore a contraction of $\overline{\Delta_s}$ into
itself. The equation has a unique solution in $\Delta_s$, and the
graphs cover $\Delta_{s/2}\times\Delta_{\rho/2}$.

\medskip
\noindent\emph{Open image and transverse coordinates.}
Consider
\[
 \Phi:\Delta_s\times\Delta_{\rho/2}\longrightarrow\C^2,
 \qquad\Phi(\xi,w)=(h_\xi(w),w).
\]
From \eqref{eq:gBound} and a Cauchy estimate, for
$|w|\le\rho/2$,
\begin{equation}\label{eq:wslope}
 |h_\xi'(w)|\le4B.
\end{equation}
Thus \eqref{eq:biholo} and \eqref{eq:wslope} imply a local
Lipschitz bound for $\Phi$. If $z=h_\xi(w)$ and
$z'=h_\eta(w')$, then
\[
 (1-\kappa)|\xi-\eta|
 \le|z-z'|+4B|w-w'|,
\]
which also bounds the inverse map in the Lipschitz norm.

To see directly that $W=\Phi(\Delta_s\times\Delta_{\rho/2})$
is open, fix an interior parameter $(\xi_0,w_0)$ and a small
closed disk about $\xi_0$. Whenever $(z,w)$ is sufficiently
close to $(h_{\xi_0}(w_0),w_0)$, the map
$\xi\mapsto z-g_\xi(w/\rho)$ sends that closed disk into
itself and remains a contraction. Hence a neighborhood of the
image point lies in $W$. It follows that $\Phi$ is a locally
bi-Lipschitz homeomorphism onto an open image.

By \eqref{eq:harmonicfinal}, every holomorphic graph is tangent
to $\ker\ddc u$. Since the latter is one-dimensional over
$\C$ throughout the chosen positive-rank neighborhood, its
tangent line equals the kernel. Finally, restoring the
pluriharmonic polynomial subtracted in Section~\ref{sec:normalization}
preserves harmonicity along each leaf. This proves the existence,
covering, and bi-Lipschitz assertions in Theorem~\ref{thm:main}(i).

\subsection{Uniqueness of holomorphic leaf germs}\label{sec:uniqueness}

The original null-slope coefficient $q=u_{z\bar w}/u_{z\bar z}$
is a priori only continuous. The constructed flow box improves its
regularity in the transverse variable, which is what makes local
uniqueness possible.

\begin{proposition}\label{prop:transverse-lip}
Let $p\in\mathcal R$. In local holomorphic coordinates $(z,w)$
with $u_{z\bar z}>0$, there exists a smaller neighborhood $W_p$
and a constant $C_p$ such that
\begin{equation}\label{eq:qLip}
 |q(z,w)-q(z',w)|\le C_p|z-z'|,
 \qquad q=\frac{u_{z\bar w}}{u_{z\bar z}},
\end{equation}
whenever $(z,w),(z',w)\in W_p$ lie in the same sufficiently small
transverse slice. In particular, the bound is uniform in $w$ on
compactly contained flow-box slices.
\end{proposition}

\begin{proof}
Applying \eqref{eq:hardyder} to $g_\xi-g_\eta$ gives, for
$|w|\le\rho/2$,
\begin{equation}\label{eq:slopeparam}
 |h_\xi'(w)-h_\eta'(w)|
 \le\frac1\rho\sqrt{\frac{80}{27}}\,
       \frac\lambda{1-\lambda}|\xi-\eta|.
\end{equation}
By \eqref{eq:graphLevi} and harmonicity along the constructed leaves,
\[
 q(h_\xi(w),w)=-\overline{h_\xi'(w)}.
\]
Since each transverse slice of the flow box is covered by the
points $h_\xi(w)$, and \eqref{eq:biholo} controls the inverse
parameterization, combining these formulas yields
\eqref{eq:qLip}, uniformly on a smaller flow box.
\end{proof}

In normalized coordinates these estimates give
$C_p\le C(\lambda)/\rho$ on the smaller flow box. The local
radius and constants depend on the modulus of continuity of $D^2u$
and the normalizing coordinates.
This proves the transverse regularity assertion of
Theorem~\ref{thm:main}(ii).

\begin{corollary}\label{cor:unique}
In the positive-rank neighborhood of Theorem~\ref{thm:main},
any two nondegenerate holomorphic curve germs through the same
point, on which $u$ is harmonic, coincide.
\end{corollary}

\begin{proof}
Since $a>0$, the $w$ component of the kernel direction does
not vanish. Every such curve germ can therefore be written as
$z=k(w)$. By \eqref{eq:graphLevi}, it satisfies
\eqref{eq:leafode}. Restrict to a horizontal real interval
$w=x+it_0$ through the intersection point. The two graphs solve
the same real ODE for a complex-valued unknown, whose right-hand
side is locally Lipschitz in the unknown by
Proposition~\ref{prop:transverse-lip}.
The Gronwall uniqueness theorem implies that they agree on a
short real interval. The holomorphic identity theorem then
implies equality of the germs.
\end{proof}

By Corollary~\ref{cor:unique}, two plaques from overlapping
positive-rank flow boxes that pass through the same point define the
same local holomorphic curve germ. Hence the flow-box foliations are
compatible on their overlaps, proving the last assertion of
Theorem~\ref{thm:main}(i).

\section{Transverse measures and laminarity}\label{sec:currents}
\subsection{The transverse measure in a flow box}

We now prove \eqref{eq:currentintro} separately from the leaf
existence statement. Let $T=\ddc u$, and let $W$ and $\Phi$ be as
in Section~\ref{sec:foliation}. We regard a $(1,1)$-current as
the current defined by a real two-form; the complex structure of
the leaves will subsequently provide positivity.

\begin{lemma}
\label{lem:pullback}
Let $T$ be a closed real two-form with continuous coefficients
on an open set $W$, and let $\Phi:V\to W$ be a locally bi-Lipschitz
homeomorphism between open subsets of $\mathbb R^4$. Define
$\Phi^*T$ using the almost-everywhere differential $D\Phi$.
Then $\Phi^*T$ has locally bounded coefficients, and
$d(\Phi^*T)=0$ in the distributional sense on $V$.
\end{lemma}

\begin{proof}
Fix $V'\Subset V''\Subset V$. The compact set
$\Phi(\overline{V''})$ lies in $W$ at positive distance from
$\partial W$. Convolving the coefficient functions of $T$ on an
open neighborhood of this compact set, with convolution radius
smaller than that distance, gives smooth forms $T_\delta$
converging uniformly to $T$ there. Convolution commutes with $d$,
so $dT_\delta=0$ wherever the convolution is defined.
For fixed sufficiently small $\delta>0$, mollify the coordinate
components of the locally Lipschitz map $\Phi$ on $V''$ to obtain
smooth maps $\Phi_j$ on $V'$ satisfying
\[
 \Phi_j\to\Phi\ \text{uniformly},\qquad
 \sup_j\|D\Phi_j\|_{L^\infty(V')}<\infty,\qquad
 D\Phi_j\to D\Phi\ \text{in }L^p(V')
\]
for every finite $p$. Uniform convergence places $\Phi_j(\overline{V'})$ in the
same open neighborhood of $\Phi(\overline{V''})$ where $T_\delta$
is defined, for all sufficiently large $j$. Each coefficient of the
pullback is a sum of terms of the form
\[
 (T_\delta)_{ab}(\Phi_j)
 \bigl(\partial_i(\Phi_j)^a\,\partial_\ell(\Phi_j)^b
       -\partial_\ell(\Phi_j)^a\,\partial_i(\Phi_j)^b\bigr).
\]
The uniform convergence of $\Phi_j$, the uniform bounds on
$D\Phi_j$, and the strong $L^2$ convergence of the derivatives
show that each product of two derivatives converges in $L^1$;
the continuous coefficients $T_\delta(\Phi_j)$ converge uniformly.
Hence $\Phi_j^*T_\delta\to\Phi^*T_\delta$ in $L^1(V')$.
For smooth maps and forms, $d(\Phi_j^*T_\delta)=0$.
Passing to distributions yields $d(\Phi^*T_\delta)=0$.
Finally, the $L^\infty$ bound for $D\Phi$ gives
\[
 \|\Phi^*T_\delta-\Phi^*T\|_{L^1(V')}
 \le C_{V'}\|T_\delta-T\|_{C^0(\overline{\Phi(V')})}
 \longrightarrow0.
\]
Another distributional limit proves $d(\Phi^*T)=0$ on $V'$.
Since $V'$ was arbitrary, the lemma follows.
\end{proof}

Write $\xi=x+iy$ and $w=r+i\tau$. Derivatives in the $r,\tau$
directions are tangent to the leaves and hence lie in the
kernel of $T$. Thus every coefficient of $\Phi^*T$ containing
$dr$ or $d\tau$ vanishes. There exists a real-valued
$A\in L^\infty_{\loc}$ such that
\begin{equation}\label{eq:pullA}
 \Phi^*T=A(\xi,w)\,dx\wedge dy.
\end{equation}
Lemma~\ref{lem:pullback} implies
\[
 \partial_rA=\partial_\tau A=0
\]
in the sense of distributions. Consequently, $A=A(\xi)$
almost everywhere.

\medskip
\noindent\emph{Identification of the transverse density.}
It is important not to assign a value on the measure-zero
slice $w=0$ to an arbitrary $L^\infty$ representative of $A$.
For almost every $(\xi,w)$, the transverse coefficient of
$\Phi^*T$ is
\begin{equation}\label{eq:jacA}
 A(\xi,w)=2a(h_\xi(w),w)\,\det_{\R}D_\xi h_\xi(w).
\end{equation}
There is a stronger parameter estimate available. By
\eqref{eq:hardypoint} and \eqref{eq:paramH2},
\begin{equation}\label{eq:identitysmall}
 \Lip_\xi\bigl(h_\xi(w)-\xi\bigr)
 \le\frac\lambda{1-\lambda}
       \frac{|w|/\rho}{\sqrt{1-|w|^2/\rho^2}}
 \longrightarrow0\qquad(w\to0).
\end{equation}
For each fixed $w$, Rademacher's theorem and the defining
property of the Lipschitz constant give
\[
 \mathop{\mathrm{ess\,sup}}_{\xi}
 \|D_\xi h_\xi(w)-I\|
 \le\Lip_\xi(h_\xi(w)-\xi).
\]
The right-hand side tends to zero. Consequently
$\det_{\R}D_\xi h_\xi(w)\to1$ in uniform essential bounds.
Meanwhile $h_\xi(w)\to\xi$ locally uniformly and $a$ is
continuous, so $a(h_\xi(w),w)\to a(\xi,0)$ uniformly on
compact transverse disks. By Fubini's theorem we may choose
$w_j\to0$ such that $A(\xi,w_j)=A(\xi)$ and
\eqref{eq:jacA} hold for almost every $\xi$ on each chosen slice.
The right side of \eqref{eq:jacA} converges to
$2a(\xi,0)$ in $L^1$ on compact transverse disks; the left
side is the same function $A(\xi)$ for all these slices.
Therefore
\begin{equation}\label{eq:Atarget}
 A(\xi)=2a(\xi,0)\quad\text{for almost every }\xi.
\end{equation}
Thus
\begin{equation}\label{eq:pullTfinal}
 \Phi^*T=i\,a(\xi,0)\,d\xi\wedge d\bar\xi.
\end{equation}

\begin{proposition}
\label{prop:laminar}
On $W$,
\begin{equation}\label{eq:laminar}
 T=\int_{\Delta_s}[\Gamma_\xi] \,2a(\xi,0)\area(\xi).
\end{equation}
\end{proposition}

\begin{proof}
Let $\psi$ be a smooth compactly supported real two-form on $W$.
The map $\Phi$ preserves the real orientation: on each
transverse slice it is a perturbation of the identity with
Lipschitz constant less than one, and its $w$ component is
unchanged. The bi-Lipschitz change-of-variables formula,
\eqref{eq:pullTfinal}, and Fubini's theorem give
\begin{align*}
 \langle T,\psi\rangle
 &=\int_W T\wedge\psi\\
 &=\int_{\Delta_s\times\Delta_{\rho/2}}
       i\,a(\xi,0)\,d\xi\wedge d\bar\xi\wedge\Phi^*\psi\\
 &=\int_{\Delta_s}2a(\xi,0)
       \left(\int_{\Gamma_\xi}\psi\right)\area(\xi).
\end{align*}
This is \eqref{eq:laminar}. The graphs are disjoint, are
defined over a common disk, and have uniformly bounded slopes,
so the integral is locally finite. Each holomorphic plaque has
no boundary inside $W$, hence its integration current is
positive and closed there.
\end{proof}

To realize Definition~\ref{def:uniform}, restrict to
$W_* = \Phi(\Delta_{s/2}\times\Delta_{\rho/4})\Subset W$.
By \eqref{eq:Hrange} and \eqref{eq:biholo}, the graphs over
$\Delta_{\rho/2}$ with $|\xi|\le3s/4$ are disjoint in the common
larger bidisk $\Delta_{R/8}\times\Delta_{\rho/2}$ and restrict to
\eqref{eq:laminar} on $W_*$. Their holomorphic extensions to $\Delta_\rho$ provide the
Cauchy estimates.

To verify compatibility, consider overlapping flow boxes $\Phi_j$,
$j=1,2$. Let $a_j$ be the transverse Levi coefficient in the $j$th
adapted chart and set $\alpha_j(\xi)=a_j(\xi,0)$. Formula
\eqref{eq:pullTfinal} gives
$\Phi_j^*T=2\alpha_j(\xi)\,d(\operatorname{Re}\xi)
\wedge d(\operatorname{Im}\xi)$.
After shrinking to a product neighborhood in the overlap, uniqueness
of leaf germs writes $F=\Phi_2^{-1}\circ\Phi_1$ as
$F(\xi,w)=(H(\xi),\psi(\xi,w))$. The holonomy map $H$ and its
inverse are locally Lipschitz by restriction of the transverse
components of these transition maps to fixed leaf-coordinate slices.
Since both flow boxes preserve real orientation and the leafwise
transition is a local biholomorphism, $H$ preserves transverse
orientation.

Bi-Lipschitz maps preserve null sets in both directions, so the
almost-everywhere chain rule gives $\Phi_1^*T=F^*(\Phi_2^*T)$.
Comparing transverse coefficients and applying Fubini's theorem,
\begin{equation}\label{eq:holonomy-density}
 \alpha_1(\xi)
 =\alpha_2(H(\xi))\,\bigl|\det_{\R}DH(\xi)\bigr|
 \quad\text{for almost every }\xi.
\end{equation}
The change-of-variables formula therefore yields
\begin{equation}\label{eq:holonomy-measure}
 H_*\bigl(2\alpha_1(\xi)\,dA(\xi)\bigr)
   =2\alpha_2(\eta)\,dA(\eta)
\end{equation}
on the corresponding local transversals. By composition this proves
holonomy invariance. The local transverse measures are thus compatible
and $T|_{\mathcal R}$ is uniformly laminar, completing the proof of
Theorem~\ref{thm:main}(iii).

\subsection{The laminarity corollary}

\begin{lemma}[Full-trace-mass criterion]\label{lem:fullmasscriterion}
Let $S$ be a positive closed $(1,1)$-current on $U\subset\C^2$.
If an open set $V\subset U$ has full $S$-trace mass and
$S|_V$ is locally uniformly laminar, then $S$ is laminar on $U$
in the sense of Definition~\ref{def:laminar}.
\end{lemma}

\begin{proof}
For a relatively compact open $U_0\Subset U$, choose
$U_\varepsilon=U_0\cap V$ and
$S_\varepsilon=S|_{U_\varepsilon}$. The coefficients of a
positive current are absolutely continuous with respect to its
trace measure. Since $S$ gives no trace mass to $U\setminus V$,
the zero extension of $S_\varepsilon$ to $U_0$ equals $S|_{U_0}$.
Hence the difference in \eqref{eq:deflaminar} has zero trace mass.
\end{proof}

\begin{proof}[Proof of Corollary~\ref{cor:fullmass}]
By Theorem~\ref{thm:main}(iii), $T|_{\mathcal R}$ is uniformly
laminar. Since $u\in C^2$, the trace measure $\mu_T=T\wedge\omega_0$ has
a continuous nonnegative density proportional to
$\operatorname{tr}H_u$. That density vanishes identically on the
rank-zero set $\mathcal Z=\Omega\setminus\mathcal R$, so
\[
 \mu_T(\mathcal Z)=0.
\]
Lemma~\ref{lem:fullmasscriterion}, with $S=T$ and $V=\mathcal R$,
therefore shows that $T$ is laminar on $\Omega$.
\end{proof}

\subsection{A smooth example}\label{sec:example}

\begin{proof}[Proof of Proposition~\ref{prop:nonuniform}]
Let $\Omega$ be a bidisk centered at the origin and set
$u(z,w)=|zw|^2$ and $T=\ddc u$. A direct calculation gives
\[
 T=i\,(w\,dz+z\,dw)\wedge(\bar w\,d\bar z+\bar z\,d\bar w),
 \qquad
 H_u=\begin{pmatrix}|w|^2&w\bar z\\z\bar w&|z|^2\end{pmatrix}.
\]
Thus $u$ is smooth and plurisubharmonic, $T\wedge T=0$,
and $T$ has rank one on $\Omega\setminus\{0\}$ and rank zero at
$0$. The trace measure has continuous density proportional to
$|z|^2+|w|^2$, so $\supp T=\Omega$ and
$(T\wedge\omega_0)(\{0\})=0$.

On $\Omega\setminus\{0\}$, the level sets $zw=c$ locally form
a holomorphic foliation, since $w\,dz+z\,dw$ is nowhere zero there.
They give locally uniformly laminar representations for $T$.
Since the rank-zero set has zero trace mass,
Lemma~\ref{lem:fullmasscriterion} shows that $T$ is laminar on
$\Omega$. More explicitly, the identity of currents
\[
 T=2\int_{\C}[\{zw=c\}\cap\Omega] \,dA(c)
\]
holds locally on $\Omega$; the singular level $c=0$ has zero
transverse measure.

We show next that $T$ is not uniformly laminar in any neighborhood
of the origin. Suppose to the contrary that there is such a
neighborhood. On a smaller flow box, uniform laminarity gives
\[
 T=\int_K[\Gamma_a] \,d\nu(a),\qquad K=\supp\nu,
\]
where the plaques are disjoint holomorphic graphs whose union is
$\supp T$ in the flow box. Their tangent lines define a continuous
line field $L$ there: continuity of the graphs in the transverse
parameter, together with Cauchy estimates on a common larger disk,
gives continuity of their derivatives.

Since $T$ has continuous coefficients, we may wedge it with each
plaque current. Positivity and
\[
 0=T\wedge T=\int_K T\wedge[\Gamma_a] \,d\nu(a)
\]
show that the restriction of $T$ to $\Gamma_a$ vanishes for
$\nu$-almost every $a$. Continuity along the plaques upgrades the
vanishing to every point of each such plaque. These parameters are
dense in $K$, so continuity of $T$ and of the plaque tangents
implies $T_x|_{L_x}=0$ at every point in the smaller flow box.
At positive-rank points, therefore, $L_x=\ker T(x)$.

But on the positive-rank locus
\[
 \ker T(z,w)=\{(v_1,v_2):wv_1+zv_2=0\}=\C(z,-w).
\]
In particular,
\[
 \ker T(z,0)=\C(1,0)\quad(z\ne0),\qquad
 \ker T(0,w)=\C(0,1)\quad(w\ne0).
\]
These two lines have different limits as $(z,w)\to(0,0)$ along
the respective coordinate axes. Hence the null-line field cannot
extend continuously through the origin, contradicting the
continuity of $L$. This proves the failure of uniform laminarity.
\end{proof}

\begin{remark}
The example has
$\operatorname{Int}(\supp T)=\Omega$ but $\rank T(0)=0$:
points in the interior of the support of a continuous positive
form need not have positive rank. In the proof of
\cite[Theorem~3.5]{DG}, rank one is used throughout the interior
of the support. This is a minor oversight in that argument;
the classical Frobenius conclusion applies on the positive-rank
locus.
\end{remark}

\subsection*{Declaration on the use of AI assistance}
ChatGPT assisted with the exploration of proof ideas and with the
drafting and editing of this manuscript. The authors take full
responsibility for the mathematical results, all proofs, and the
final content of the paper.

\end{document}